\documentclass{scrartcl}

\usepackage{natbib}
\usepackage{geometry}
\usepackage{fleqn}
\usepackage{graphicx}
\usepackage[colorlinks,citecolor=blue,urlcolor=blue]{hyperref}
\usepackage{xspace,enumerate} 
\usepackage{amsmath,amssymb}
\usepackage{enumitem}
\usepackage{cite}
\usepackage{mathtools}
\usepackage{dsfont,mathrsfs}
\usepackage{amsthm}
\usepackage{bbm}

\numberwithin{equation}{section}

\theoremstyle{plain}
\newtheorem{theorem}{Theorem}[section]
\numberwithin{theorem}{section}

\newtheorem{definition}[theorem]{Definition}

\newtheorem{remark}[theorem]{Remark}

\newcommand{\cD}{\mathcal{D}}
\newcommand{\cF}{\mathcal{F}}
\newcommand{\cL}{\mathcal{L}}

\newcommand{\RR}{\mathbb{R}}

\DeclareMathOperator{\supp}{supp}
\DeclareMathOperator{\id}{id}

\title{Comparison of Markov processes by the martingale comparison method}
\author{Benedikt K\"opfer\footnote{A LGFG grant of the state Baden-W\"urttemberg is gratefully acknowledged}, Ludger R\"uschendorf}
\date{}

\begin{document}

\maketitle
\thispagestyle{empty}

\begin{abstract}
Comparison results for Markov processes w.r.t. function class induced (integral) stochastic orders have a long history. The most general results so far for this problem have been obtained based on the theory of evolution systems on Banach spaces. In this paper we transfer the martingale comparison method, known for the comparison of semimartingales to Markovian semimartingales, to general Markov processes. The basic step of this martingale approach is the derivation of the supermartingale property of the linking process, giving a link between the processes to be compared. In this paper this property is achieved using in an essential way the characterization of Markov processes by the martingale problem. As a result the martingale comparison method gives a comparison result for Markov processes under a general alternative set of regularity conditions compared to the evolution system approach.
\end{abstract}

\renewcommand{\thefootnote}{}
\footnotetext{\hspace*{-.51cm}%
AMS 2010 subject classification:
Primary: 60E15; Secondary: 60J25.\\
Key words and phrases: Ordering of Markov processes, martingale comparison method, evolution systems.}

{\renewcommand{\thefootnote}{\arabic{footnote}}

\section{Evolution systems and comparison of Markov processes}
\label{sec:evscomp}

Stochastic ordering and comparison results for Markov processes are a basic problem of probability theory. They have a long history and are motivated by a number of applications in a variety of fields (see \citet{Ma87}, \citet{Co96}, \citet{DS06}, \citet{Ru08}, \citet{KM09}, \citet{RW11}, \citet{RWS16}, \citet{Cr17} and \citet{Cr19}). Various approaches ranging from analytic to coupling methods have been developed to this aim sometimes in the context of specific models or specific applications. The most general comparison results so far have been obtained based on the theory of evolution systems on Banach spaces (see \citet{RWS16}).

The transition operators $T_{s,t}$, $s\le t$, of a Markov process $X$ with values in a metric space $S$ are an evolution system on the space of measurable bounded real-valued functions $\mathcal{L}_b(S)$. Since the transition operators are defined by conditional expectations it is possible to consider them also on function spaces different from $\mathcal{L}_b(S)$. In order to stay in the framework of evolution systems, we consider the transition operators on Banach spaces. We also assume that the Banach spaces in use consist of integrable functions in the sense that they are integrable with respect to all conditional laws. Generally a family of bounded linear operators $(T_{s,t})_{s \le t}$ from a Banach space $\mathbbm{B}$ to $\mathbbm{B}$ is called an \emph{evolution system} if for all $0 \le s \le t \le u$ holds
\begin{enumerate}
\item $T_{s,s} = \id$,
\item $T_{s,u} = T_{s,t} T_{t,u}$.
\end{enumerate}
An evolution system is called strongly continuous if for all $f \in \mathbbm{B}$ the $\mathbbm{B}$-valued function $(s,t) \mapsto T_{s,t}f$ is continuous.
If the evolution system is time-homogeneous, i.e. it only depends on the duration $t-s$, then $(T_t)_{t \ge 0}$ defined by $T_t f := T_{0,t} f$ is a semigroup. An evolution system $(T_{s,t})_{s \le t}$ is called a Feller evolution system if it is strongly continuous and maps $C_0(S)$ into itself. If the evolution system maps $C_b(S)$ into itself, we call it a $C_b$-Feller evolution system. Further, if the transition operators of a Markov process $X$ are a ($C_b$-)Feller evolution system, $X$ is called a ($C_b$-)Feller process.

Left and right generators of evolution systems $(T_{s,t})$ on a Banach space $\mathbbm{B}$ are defined by:
\begin{align*}
A_s^+f := \lim_{h \downarrow 0} \frac{T_{s,s+h}f - f}{h} ~~~\text{for all}~ s \in \RR_+.
\end{align*}
This operator is defined on its domain $\cD(A^+_s)$, i.e. for all $f \in \mathbbm{B}$ for which the limit exists in the norm. Analog we define the left generators on the domain $\cD(A^-_s)$ by
\begin{align*}
A_s^-f := \lim_{h \downarrow 0} \frac{T_{s-h,s}f - f}{h} ~~~\text{for all}~ s \in \RR_+ \setminus \{0\}.
\end{align*}
If we weaken the limit in the definitions of left and right generators to a pointwise limit, then the corresponding operators are called extended pointwise right and left generators (see \citet{GC06}). The generators of an evolution system on a Banach space $\mathbbm{B}$ are linear operators on $\mathbbm{B}$. In general the right generator and left generator do not coincide. In \citet{Bo14} an explicit example for a Markov process is given whose right and left generators do not coincide. There also a condition is given to imply equality for the left and right generators.

Evolution systems arise as solutions of homogeneous evolution problems. Let $(T_{s,t})_{s \le t \le T}$ be an evolution system on some Banach space $\mathbbm{B}$. We set
\begin{align*}
\cD_+(t) := \{ f \in \mathbbm{B} ; s \mapsto T_{s,t}f ~\text{is differentiable from the right on } (0,t)\},
\end{align*}
and
\begin{align*}
\cD_+^A(s) := \bigcap_{s \le t} \cD(A_t), ~~ \cD^A(s,t) := \bigcap_{s \le u \le t} \cD(A_u).
\end{align*}
The following theorem restates basic connections of evolution systems to their right generators from \citet{GC06} and states some corresponding representation results. 

\begin{theorem}
\label{thm:forbaceqti}
Let $(T_{s,t})_{s \le t \le T}$ be an evolution system on a Banach space $\mathbbm{B}$ with right generators $(A^+_t)_{t \in [0,T)}$. Then the following assertions hold true:
\begin{enumerate}
\item If $(T_{s,t})_{s\le t \le T}$ is strongly continuous, then for fixed $t$ the function $u: s \mapsto T_{s,t}f$ with $f \in \cD_+(t)$ is a solution to the following final value problem on $(0,t)$
\begin{align}
\label{eq:evsback}
\begin{cases}
\frac{\partial^+}{\partial s} u(s) &= -A^+_s u(s),\\
\lim_{s \uparrow t} u(s) &= f.
\end{cases}
\end{align}
		
\item For $f \in \cD_+^{A^+}(s)$, fixed $0 < s < T$ and for every $s < t < T$ the forward equation holds:
\begin{align}
\label{eq:evsforward}
\frac{\partial^+}{\partial t} T_{s,t}f = T_{s,t}A^+_tf.
\end{align}

\item Representation results: Let $(T_{s,t})_{s\le t \le T}$ is strongly continuous and $f \in \cD_+(t)$. Further, assume that the right derivative $\frac{\partial^+}{\partial u} T_{u,t}f$ is integrable on $[s,t]$. Then the following integral representation of the evolution system holds true
\begin{align*}
T_{s,t}f - f = \int_s^t A^+_u T_{u,t} f du.
\end{align*}
If $f \in \cD^{A^+}(s,t)$ and the right derivative $\frac{\partial^+}{\partial u} T_{s,u} f$ is integrable on $[s,t]$, then it holds that
\begin{align}
\label{eq:darsttoright}
T_{s,t}f - f = \int_s^t T_{s,u}A^+_uf ~du.
\end{align}
\end{enumerate}
\end{theorem}

A similar integral representation also holds true for $f \in \cD^{A^-}(s,t)$ and the left derivative.

\begin{theorem}
\label{thm:backti}
Let $(T_{s,t})_{s \le t \le T}$ be an evolution system on some Banach space $\mathbbm{B}$ with left generators $(A^-_t)_{t \in (0,T]}$ and define 
\begin{align*}
\cD_-(t) := \{ f \in \mathbbm{B}; s \mapsto T_{s,t} \text{ is differentiable from the left on } (0,t)\}
\end{align*}
and
\begin{align*}
F_t := \left \{ f \in \mathbbm{B}; \lim_{h \downarrow 0} T_{t-h,t}f = f \right \}.
\end{align*}
Then it holds that:
\begin{enumerate}
\item For $0 < t < T$ and $f \in \cD_-(t) \cap F_t$, the function $u: s \mapsto T_{s,t}f$ is a solution to the backward equation on $(0,t)$
\begin{align*}
\begin{cases}
\frac{\partial^-}{\partial s} u(s) &= -A^-_s u(s),\\
\lim_{s \uparrow t} u(s) &= f.
\end{cases}
\end{align*}

\item If $(T_{s,t})_{s\le t \le T}$ is strongly continuous in $s$, $f \in \cD_-(t)$ and the left derivative $\frac{\partial^-}{\partial s} T_{s,t}f$ is integrable on $[s,t]$, we have the following integral representation of the evolution system
\begin{align*}
T_{s,t}f - f = \int_s^t A^-_u T_{u,t} f du.
\end{align*}
\end{enumerate}
\end{theorem}

For further extensions and properties of the notion of left (right) generators see \citet{Ca11}, \citet{Bo14} and \citet{Ko19}.

A basic result in the theory of evolution systems is the following integral representation for solutions to an inhomogeneous evolution problem (see \citet{RWS16}). For fixed $s,t \in \RR_+$ let $(A_r)_{s \le r \le t}$ be a family of linear operators on a Banach space $\mathbbm{B}$ and let $G : [s,t] \to \mathbbm{B}$. A function $u: [s,t] \to \mathbbm{B}$, right differentiable on $(s,t)$ and such that $u(r) \in \cD(A_r)$ for all $s \le r < t$, is called a solution to the \emph{inhomogeneous right evolution problem} with boundary condition $f \in \mathbbm{B}$ if
\begin{align*}
\frac{\partial^+}{\partial r} u(r) =&~ -A_r u(r) + G(r)~~\text{for}~s<r\le t,\\
u(t) =&~ f.
\end{align*} 
If additionally $u$ is continuous on $[s,t]$ it is called a classical solution to the inhomogeneous right evolution problem.\\
On the other hand for $u: [s,t] \to \mathbbm{B}$, left differentiable on $(s,t)$ such that $u(r) \in \cD(A_r)$ for all $s < r \le t$, $u$ is a solution to the inhomogeneous left evolution problem with boundary condition $f \in \mathbbm{B}$ if
\begin{align*}
\frac{\partial^-}{\partial r} u(r) =&~ -A_r u(r) + G(r)~~\text{for}~s < r \le t,\\
u(t) =&~ f.
\end{align*} 
If $u$ is continuous it is called a classical solution.

The representation result is as follows.

\begin{theorem}
\label{thm:intdarst_weakepright}
Let $(T_{s,t})_{s \le t}$ be a strongly continuous evolution system on a Banach space $\mathbbm{B}$ with right generators $(A^+_t)_{t \ge 0}$. For fix $t \in \RR_+$ let $F_t,G : [0,t] \to \mathbbm{B}$ be such that
\begin{enumerate}
\item the function $r \mapsto T_{s,r}G(r)$ is integrable on $[s,t]$,

\item $F_t$ solves the inhomogeneous right evolution problem for the operators $(A^+_s)_{s \le t}$, 
\begin{align*}
\frac{\partial^+}{\partial r} F_t(r) =&~ -A^+_r F_t(r) + G(r).
\end{align*}
\end{enumerate} 
Then the following representation holds
\begin{align*}
F_t(s) = T_{s,t}F_t(t) - \int_s^t T_{s,r} G(r) dr.
\end{align*}
\end{theorem}

The same representation result also hold for the inhomogeneous left evolution problem for left generators of a strongly continuous evolution system (see \citet{Ko19}). The representation result is the basic tool for the general comparison theorem for Markov processes by means of evolution systems in \citet{RWS16} stating an ordering result of Markov processes w.r.t. function classes $\cF$. Therefore let $X$ and $Y$ be Markov processes with corresponding transition operators $T^X$ and $T^Y$. Under some regularity conditions this result states that a propagation of order property for $X$, i.e. $f \in \cF$ implies $T^X_{s,t}f \in \cF$ and comparison of generators implies the stochastic ordering condition $X_t \le_{\cF} Y_t$, $t \ge 0$.

The following is a reformulation of this result holding true for single functions $f$. Note that for this case where $\cF = \{ f \}$ the propagation of order property does not make sense.

\begin{theorem}
\label{thm:comprightgen}
Assume that $(T^X_{s,t})_{s \le t}$ and $(T^Y_{s,t})_{s \le t}$ are strongly continuous evolution systems on $\mathbbm{B}$ and let $f \in \mathbbm{B}$. If for fixed $t \in \RR_+$ it holds that for all $s \le t$
\begin{enumerate}
\item $T^X_{s,t}f \in \cD(A_s^{Y+})$;

\item $r \mapsto T^Y_{s,r}(A^{X+}_r - A^{Y+}_r) T^X_{s,r}f$ is integrable on $[s,t]$;

\item $A^{X+}_s T^X_{s,t}f \le A^{Y+}_s T^X_{s,t} f$ a.s..
\end{enumerate}
Then it holds that
\begin{align*}
T^X_{s,t} f \le T^Y_{s,t} f ~~\text{a.s. for all } s \le t.
\end{align*}
\end{theorem}

\begin{proof}
Set $F_t(s) := T^Y_{s,t}f - T^X_{s,t}f$, then $F_t$ satisfies the equation
\begin{align*}
\frac{\partial^+}{\partial s} F_t(s) = - A^{Y+}_s T^Y_{s,t}f + A^{X+}_s T^X_{s,t}f,
\end{align*}
with boundary condition $F_t(t) = 0$. This equation can be written as 
\begin{align}
\label{eq:ep_proof_compright}
\frac{\partial^+}{\partial s} F_t(s) = - A^{Y+}_s (T^Y_{s,t}f - T^X_{s,t}f) + (A^{X+}_s - A^{Y+}_s) T^X_{s,t}f =: - A^{Y+}_s F_t(s) + G(s),
\end{align}
here $G(s):= (A^{X+}_s - A^{Y+}_s) T^X_{s,t}f$. The terms in the equation are well defined by Theorem \ref{thm:forbaceqti} and Assumption~$1.$ Hence, $F_t$ solves an inhomogeneous right evolution problem.\\
From the strong continuity of the evolution systems we deduce that $F_t$ is continuous in $s$. Hence, $F_t$ is a classical solution to the inhomogeneous right evolution problem \eqref{eq:ep_proof_compright}.\\
We show that $F_t$ is nonnegative; then the assertion follows. To see this we apply the integral representation in Theorem \ref{thm:intdarst_weakepright} to $F_t$ to obtain
\begin{align*}
F_t(s) = T^Y_{s,t} F_t(t) - \int_s^t T^Y_{s,r} G(r) dr = \int_s^t T^Y_{s,r} (-G(r)) dr.
\end{align*}
From Assumption~$3.$ it follows that $-G(r) \ge 0$ a.s. and hence the assertion follows from the fact that the transition operators of Markov processes are positivity preserving operators.
\end{proof}

A similar comparison result also holds for left generators (see \citet{Ko19}).

\section{The martingale comparison method for Markov processes}

For the comparison of a semimartingale $X$ to a Markovian semimartingale $Y$ \citet{GM02} introduced the martingale comparison method. The basic step of this approach is to establish that the \emph{linking process}
\begin{align*}
(T^X_{s,t}f(Y_s))_{0 \le s \le t}
\end{align*}

is a supermartingale for fixed $t$. Note that $T^X_{t,t}f(Y_t) = E[f(Y_t]$ and $T_{0,t}f(x_0) = E[f(X_t)]$ assuming that $X_0 = x_0 = Y_0$. Thus $(T_{s,t}f(Y_s))$ gives a link between the processes $X$ and $Y$. From the supermartingale property of the linking process as a direct consequence the following comparison result is obtained:
\begin{align}
\label{eq:supmgineq}
E[f(Y_t)] = E[T^X_{t,t}f(Y_t)] \le T^X_{0,t}f(x_0) = E[f(X_t)].
\end{align}
If $(T^X_{s,t}f(Y_s))_{0 \le s \le t}$ is a submartingale, the reverse inequality holds. The proof of the supermartingale property is essentially based on It\^o's formula and on a version of Kolmogorov backwards equation for Markovian semimartingales. In this paper we transfer this martingale comparison approach to the comparison of general Markov processes. As main tool we make essentially use of the characterization of Markov processes by the martingale problem. We transfer this classical result (see e.g. \citet[Ch.4., Prop. 1.7]{EK05}) to the frame of Markov processes with transition operators defined on a Banach spaces $\mathbbm{B}$ of integrable functions; for detailed exposition see \citet{Ko19}. 

\begin{theorem}
\label{thm:zshgmarkovmartingale}
Let $(X_t)_{t \in [0,T]}$ be a Markov process with strongly continuous transition operators $(T_{s,t})_{s \le t \le T}$ on some Banach space $\mathbbm{B}$ and corresponding right generators $(A^+_t)_{t \in [0,T)}$. If for $f \in \cD^{A+}_+(0)$ the right derivative $\frac{\partial^+}{\partial t} T_{s,t}f$ is integrable on $[0,T)$, it holds that the process $(M_t)_{t \in [0,T]}$ defined by
\begin{align*}
M_t := f(X_t) - f(X_0) - \int_0^t A^+_s f(X_s) ~ds 
\end{align*}
is an $(\cF_t)_{t \ge 0}$ martingale.
\end{theorem}

\begin{proof}
The integrability is clear since the Banach space is assumed to consist of integrable functions and the right generator $(A^+_t)_{t \in [0,T)}$ maps the Banach space into itself. Let $0 \le s \le t$, then we have by the Markov property and equation \eqref{eq:darsttoright}
\begin{align*}
E[M_t | \cF_s] =~& E[f(X_t)|X_s] - f(X_0) - \int_0^t E[A^+_u f(X_u)|X_s] ~du \\
=~& T_{s,t}f(X_s) - f(X_0) - \int_s^t T_{s,u} A^+_u f(X_s)~du - \int_0^s A^+_u f(X_u) ~du\\
=~& f(X_s) - f(X_0) - \int_0^s A^+_u f(X_u) ~du\\
=~& M_s.
\end{align*}
This shows the assertion.
\end{proof}

A similar martingale property also holds for the left generators. We also will make use of the martingale property for space time functions $f(t,x)$. To that aim we state the following definition, a variant of the definition in \citet{Ca11} for general Banach spaces. The family of operators $(A^+_t)_{t \in [0,T)}$ is here regarded as single operator $A$ on a bigger space consisting of functions $f: [0,T) \times S \to \RR$. Therefore it is important that the Banach space on which each $A_t$ is defined can be extended resonably to functions of the space time process, like $L^p(\RR^d)$, $\cL_b(\RR^d)$ and the smooth functions vanishing at infinity $C_0^\infty(\RR^d)$.

\begin{definition}
\label{def:definitionfromCasterenBanach}
A family of operators $(A^+_t)_{t \in [0,T)}$ on some Banch space $\mathbbm{B}$ is said to be a right generator of a Markov process $X$ if for all $f \in \cD_+(A)$, for all $x \in S$ and for all $s \le t$ it holds that
\begin{align*}
\frac{\partial^+}{\partial t} E\left [f(t,X_t)|X_s = x \right ] = E \left [ \left .\frac{\partial^+}{\partial t} f(t,X_t) + A^+_t f(t,\cdot)(X_t) \right | X_s = x \right ].
\end{align*} 
A family of operators $(A^-_t)_{t \ge 0}$ on $\mathbbm{B}$ is said to be a left generator of $X$ if we replace the right derivatives above by left derivatives.
\end{definition}

We remark that also the extended pointwise right and left generators $(A^+_t)$ and $(A^-_t)$ of strongly continuous transition operators are right and left generators in the sense of Definition \ref{def:definitionfromCasterenBanach} (see \citet{Ko19}). 

With the help of this definiton we can formulate the martingale property for the space time process.
\begin{theorem}
\label{thm:zshgmarkovmartingalestatetime}
Let $(A^+_t)_{t \in [0,T)}$ be the right generator of a Markov process $X$ in the sense of Definition \ref{def:definitionfromCasterenBanach}. Then for every $f \in \cD_+(A)$ such that $\frac{\partial^+}{\partial t}E[f(t,X_t)|X_s]$ is integrable on $[0,T)$, we have that
\begin{align*}
M_t := f(t,X_t) - f(0,X_0) - \int_0^t \left (\frac{\partial^+}{\partial s} + A^+_s\right ) f(s,X_s) ds
\end{align*}
is an $(\cF_t)_{t \ge 0}$ martingale.
\end{theorem}

\begin{proof}
Again the integrability is clear and the martingale property can be shown straightforward:
\begin{align*}
E[M_t|\cF_s] =~& E[f(t,X_t)|X_s] - f(0,X_0) - \int_0^t E \left [ \left .\left (\frac{\partial^+}{\partial u} + A^+_u\right ) f(u,X_u)\right |X_s \right ] du\\
=~& E[f(t,X_t)|X_s] - f(0,X_0) - \int_s^t E \left [ \left .\left (\frac{\partial^+}{\partial u} + A^+_u\right ) f(u,X_u)\right |X_s \right ] du\\
~& - \int_0^s \left (\frac{\partial^+}{\partial u} + A^+_u\right ) f(u,X_u) du\\
=~& E[f(t,X_t)|X_s] - f(0,X_0) - \int_s^t \frac{\partial^+}{\partial u}E \left [ f(u,X_u)|X_s \right ] du\\
~& - \int_0^s \left (\frac{\partial^+}{\partial u} + A^+_u\right ) f(u,X_u) du\\
=~& E[f(t,X_t)|X_s] - f(0,X_0) - E[f(t,X_t)|X_s] + f(s,X_s)\\
&~ - \int_0^s \left (\frac{\partial^+}{\partial u} + A^+_u\right ) f(u,X_u) du\\
=~& M_s.
\end{align*}
This completes the proof.
\end{proof}

\begin{remark}
\begin{enumerate}
\item In \citet{Ca11} a similar result is given under the assumption that the function $f$ to be continuously differentiable in the time variable.

\item The proof of Theorem \ref{thm:zshgmarkovmartingalestatetime} can also be adapted for the extended pointwise right and left generators of the transition operators of $X$. Thus, Theorem \ref{thm:zshgmarkovmartingale} also holds for the extended pointwise right and left generators.
\end{enumerate}
\end{remark}

The connection of Markov processes to martingales allows the introduction of further extensions of generators. Therefore, we give some definitions which are motivated by Theorem \ref{thm:zshgmarkovmartingale}. They are variants of definitions form \citet{Ci80} for time-inhomogeneous Markov processes.

\begin{definition}
Let $(A_t)_{t \ge 0}$ be a family of operators with domains $(\cD(A_t))_{t \ge 0}$. It is called \emph{extended generator} of a Markov process $X$ if $\cD_+^A(0)$ consists of measurable functions $f: S \to \RR$ such that for all $t \ge 0$ the functions $A_tf: S \to \RR$ are measurable and
\begin{align*}
f(X_t) - f(X_0) - \int_0^t A_sf(X_s) ds
\end{align*}
is well defined and a local martingale.
\end{definition}

Note that it makes no sense to distinguish between left and right generators since here the interpretation as partial semi-differential of the underlying evolution system is not taken. Also there is no restriction to Banach spaces as domains.

The same definition can be given for the space-time process. Recall that the Banach space under consideration has to be extendable to the space-time process.

\begin{definition}
\label{def:extendedgencinlar}
Let $(A_t)_{t \ge 0}$ be family of operators with domains $(\cD(A_t))_{t \ge 0}$. It is called \emph{extended right generator} of the space-time process of $X$ if the cuts of the domains $\cD_+(A)$ consists of measurable functions $f: \RR_+ \times S \to \RR$ such that for all $t \ge 0$ the function $A_tf:\RR_+ \times S \to \RR$ is measurable and
\begin{align*}
f(t,X_t) - f(0,X_0) - \int_0^t \left (\frac{\partial^+}{\partial s} + A_s \right )f(s,X_s) ds
\end{align*}
is well defined and a local martingale.\\
If the derivatives above are replaced by left derivatives we call the corresponding family of operators extended left generators.
\end{definition}

The extended generators can be expanded to other integrals than the Lebesgue integral. This is particularly interesting for example if we consider general Markovian semimartingales with fixed jump times.

\begin{definition}
\label{def:randomgencinlar}
Let $F \in \mathscr{V}^+$ be predictable. A family of operators $(A_t)_{t \ge 0}$, $A_t: \mathcal{L}(S) \to \mathcal{L}(\Omega)$ is called \emph{$F$-random generator} of a Markov process $X$, if $\cD_+^A(0)$ consists of functions $f: \RR^d \to \RR$ for which $(A_tf)_{t \ge 0}$ is an optional process such that $Af \cdot F \in \mathscr{V}$ is predictable and 
\begin{align*}
f(X_t) - f(X_0) - \int_0^t A_sf dF_s
\end{align*}
is well defined and a local martingale.
\end{definition}

Based on the martingale problem we obtain a transfer of the martingale comparison method to the comparison of general Markov processes. 
In the following theorem we consider the transition operators $T^X_{s,t}$ for fixed $t$ and $f \in \mathbbm{B}$ as a function $T^X_{\cdot,t}f: [0,t] \times S \to \RR$. Hence, we can insert the space-time process and obtain the connection to the martingale problem from Theorem \ref{thm:zshgmarkovmartingalestatetime}. Note that we use generators in the sense of Definition \ref{def:definitionfromCasterenBanach}. For processes $X$, $Y$ we denote their (right) generators by $A^{X+}$ and $A^{Y+}$.

\begin{theorem}[Comparison by the martingale comparison method]
\label{thm:compbyprobaright}
Let $(T^X_{s,t})_{s \le t}$ and $(T^Y_{s,t})_{s \le t}$ be strongly continuous and $f \in \mathbbm{B}$. For fixed $t \in \RR_+$ assume that for all $s \le t$ the following holds
\begin{enumerate}
\item $T^X_{\cdot,t}f \in \cD_+(A^{X+}) \cap \cD_+(A^{Y+})$;

\item $\frac{\partial^+}{\partial u}E[T_{u,t}^Xf(X_u)|X_s]$ and $\frac{\partial^+}{\partial u}E[T_{u,t}^Xf(Y_u)|Y_s]$ are integrable on $[0,t]$;

\item $\supp(P^{Y_s}) \subset \supp(P^{X_s})$;

\item $A^{X+}_s T^X_{s,t} f \ge A^{Y+}_s T^X_{s,t} f$ a.s.
\end{enumerate}
Then it holds that 
\begin{align*}
E[f(Y_t)] \le E[f(X_t)].
\end{align*}
\end{theorem}

\begin{proof}
By construction $(T_{s,t}^Xf(X_s))_{s \le t}$ is a martingale; this follows by the Markov property. For $u \le s$ we have
\begin{align*}
E[T_{s,t}^Xf(X_s)|\cF_u] =&~ E[E[f(X_t)|X_s]|\cF_u]\\
=&~ E[E[f(X_t)|\cF_s]|\cF_u]\\
=&~ E[f(X_t)|\cF_u]\\
=&~ T_{u,t}^Xf(X_u).
\end{align*}
On the other hand by Assumption~$2.$
\begin{align*}
T_{s,t}^Xf(X_s) - T_{0,t}^Xf(X_0) - \int_0^s \left ( \frac{\partial^+}{\partial u} + A^{X+}_u \right ) T_{u,t}^Xf(X_u) du
\end{align*} 
is a martingale as well. It follows that the integral process
\begin{align*}
\int_0^s \left ( \frac{\partial^+}{\partial u} + A^{X+}_u \right ) T_{u,t}^Xf(X_u) du
\end{align*} 
is also a martingale starting in zero. Since it is an integral with respect to the Lebesgue measure, it is of finite variation and continuous. By \citet[Corollary I.3.16]{JS03} it follows that it is zero almost surely. Thus, the integrand must be zero $\lambda \times P$ almost surely,
\begin{align*}
\left ( \frac{\partial^+}{\partial u} + A^{X+}_u \right ) T_{u,t}^Xf(X_u) = 0.
\end{align*}
Hence, for all $x \in \supp(P^{X_u})$ except a set of Lebesgue measure zero we obtain 
\begin{align*}
\left ( \frac{\partial^+}{\partial u} + A^{X+}_u \right ) T_{u,t}^Xf(x) = 0.
\end{align*}
By Assumption~$3.$, this implies that
\begin{align}
\label{eq:zero}
\left ( \frac{\partial^+}{\partial u} + A^{X+}_u \right ) T_{u,t}^Xf(Y_u) = 0.
\end{align}
$\lambda \times P$ almost surely. Therefore by Assumption~$1.$,~$2.$ and Theorem \ref{thm:zshgmarkovmartingalestatetime} applied to $T_{s,t}^Xf$ we get that 
\begin{align*}
M_s :=T_{s,t}^Xf(Y_s) - T_{0,t}^Xf(Y_0) - \int_0^s \left ( \frac{\partial^+}{\partial u} + A^{Y+}_u \right ) T_{u,t}^Xf(Y_u) du
\end{align*}
is a martingale. Combining this with \eqref{eq:zero} this implies that the following process is a martingale
\begin{align*}
T_{s,t}^Xf(Y_s) - T_{0,t}^Xf(x_0) - \int_0^s \left ( A^{Y+}_u - A^{X+}_u\right ) T_{u,t}^Xf(Y_u) du.
\end{align*}
By Assumption~$4.$ the integral is non-positive and it follows that $(T_{s,t}^Xf(Y_s))_{s \le t}$ has the representation
\begin{align*}
T_{s,t}^Xf(Y_s) = T_{0,t}^Xf(x_0) + M_t + \int_0^s \left ( A^{Y+}_u - A^{X+}_u\right ) T_{u,t}^Xf(Y_u) du.
\end{align*}
This is a supermartingale since the integral is non-positive. The assertion then follows by inequality \eqref{eq:supmgineq}.
\end{proof}

\begin{remark}
From the proof of Theorem \ref{thm:compbyprobaright} it follows in a similar way that the inverse inequality, $A^{X+}_s T^X_{s,t}f \le A^{Y+}_s T^X_{s,t}f$, implies that $(T^X_{s,t}f(Y_s))_{0 \le s \le t}$ is a submartingale and hence the expectations are ordered the other way around $E[f(Y_t)] \ge E[f(X_t)]$.
\end{remark}

Since we also have an analog martingale property for left generators, we can transfer Theorem \ref{thm:compbyprobaright} to left generators. 

\begin{theorem}
Let $(T^X_{s,t})_{s \le t}$ and $(T^Y_{s,t})_{s \le t}$ be strongly continuous and $f \in \mathbbm{B}$. For fixed $t \in \RR_+$ assume that for all $s \le t$ we have
\begin{enumerate}
\item $T^X_{\cdot,t}f \in \cD_-(A^{Y-}) \cap \cD_-(A^{Y-})$;

\item $\supp(P^{Y_s}) \subset \supp(P^{X_s})$;

\item $\frac{\partial^-}{\partial u}E[T_{u,t}^Xf(X_u)|X_s]$ and $\frac{\partial^-}{\partial u}E[T_{u,t}^Xf(Y_u)|Y_s]$ are integrable on $[0,t]$;

\item $A^{X-}_s T^X_{s,t} f \ge A^{Y-}_s T^X_{s,t} f$ a.s.
\end{enumerate}
Then it holds that 
\begin{align*}
E[f(Y_t)] \le E[f(X_t)].
\end{align*}
\end{theorem}

\begin{proof}
The proof is similar to the proof of Theorem \ref{thm:compbyprobaright}.
\end{proof}

The extended generators and the random generators are defined by a local martingale property. Since the results above rely on the martingale property we are able to obtain similar results for the extended generators and the random generators. The main difference now is that we only have a local martingale property and we are not restricted to Banach spaces.

Let $(A^{X+}_t)_{t \ge 0}$ and $(A^{Y+}_t)_{t \ge 0}$ be the extended right generators for $X$ and $Y$, see Definition \ref{def:extendedgencinlar}. Let $f: S \to \RR$ be an integrable function such that $(T^X_{s,t}f(X_s))_{s \ge 0} \in \cD_+(A^{X+})$. By the martingale property of $(T^X_{s,t}f(X_s))_{s \ge 0}$ we get that $\lambda \times P$ almost surely
\begin{align*}
\left (\frac{\partial^+}{\partial s} + A^{X+}_s \right ) T^X_{s,t}f(X_s) = 0. 
\end{align*}
We then can undertake the same steps as in the proof of Theorem \ref{thm:compbyprobaright} which yields a local supermartingale property for $(T^X_{s,t}f(Y_s))_{s \ge 0}$. So we only need to specify the particular assumptions and make sure that $(T^X_{s,t}f(Y_s))_{s \ge 0}$ is a proper supermartingale.

\begin{theorem}
Let $f: S \to \RR$ be integrable. Assume that
\begin{enumerate}
\item $T^X_{\cdot,t}f \in \cD_+(A^{X+}) \cap \cD_+(A^{Y+})$,

\item $\supp(P^{Y_s}) \subset \supp(P^{X_s})$,

\item $(T^X_{s,t}f(Y_s)^-)_{s \ge 0}$ is of class (DL),

\item $A^{Y+}_s f \le A^{X+}_s f$ a.s.
\end{enumerate}
Then
\begin{align*}
E[f(Y_t)] \le E[f(X_t)].
\end{align*}
\end{theorem}

\begin{proof}
The proof is similar to the proof of Theorem \ref{thm:compbyprobaright}. As mentioned above we have $\lambda \times P$ almost surely
\begin{align*}
\left (\frac{\partial^+}{\partial s} + A^{X+}_s \right ) T^X_{s,t}f(X_s) = 0. 
\end{align*}
By Assumption~$2.$ we obtain that $\left (\frac{\partial}{\partial s} + A^{X+}_s \right ) T^X_{s,t}f(Y_s) = 0$ $\lambda \times P$ almost surely as well. On the other hand by the definition of the extended generator it holds that
\begin{align*}
T^X_{s,t}f(Y_s) - T^X_{0,t}f(x_0) - \int_0^s \left ( \frac{\partial^+}{\partial u} + A^{Y+}_u \right ) T^X_{u,t}f(Y_u) du
\end{align*}
is a local martingale. We substitute the time derivative:
\begin{align*}
T^X_{s,t}f(Y_s) - T^X_{0,t}f(x_0) - \int_0^s \left ( A^{Y+}_u - A^{X+}_u\right ) T^X_{u,t}f(Y_u) du.
\end{align*}
The integral is non-positive and hence $(T^X_{s,t}f(Y_s))_{0 \le s \le t}$ is a local supermartingale. By Assumption~$3.$ it is a proper supermartingale. The assertion now follows as in Theorem \ref{thm:compbyprobaright}.
\end{proof}

An analog result also holds in the case of extended left generators.

Finally, if we consider random generators instead of extended generators, we get a similar comparison result. Here we have the advantage that in Definition \ref{def:randomgencinlar} there appears no partial derivative. This means that we can proceed more directly.

\begin{theorem}
Let $F \in \mathscr{V}^+$ be predictable and $f$ be integrable. Assume that $X$ and $Y$ possess $F$-random generators $(A^X_t)_{t \ge 0}$ and $(A^Y_t)_{t \ge 0}$. Further, let
\begin{enumerate}
\item $X_0 \sim Y_0$,

\item $(f(X_t)- f(Y_t))_{t \ge 0}$ be of class (DL),

\item $A^Y_s f \le A^X_s f$ a.s.
\end{enumerate}
Then we obtain
\begin{align*}
E[f(Y_t)] \le E[f(X_t)].
\end{align*}
\end{theorem}

\begin{proof}
Since $(A^X_t)_{t \ge 0}$ and $(A^Y_t)_{t \ge 0}$ are $F$-random generators we have that 
\begin{align*}
f(X_t) - f(X_0) - \int_0^t A^X_s f dF_s 
\end{align*}
and 
\begin{align*}
f(Y_t) - f(Y_0) - \int_0^t A^Y_s f dF_s 
\end{align*}
are local martingales. It follows that 
\begin{align*}
f(X_t) - f(X_0) - \int_0^t A^X_s f dF_s - f(Y_t) + f(Y_0) + \int_0^t A^Y_s f dF_s\\
 = f(X_t) - f(X_0) - f(Y_t) + f(Y_0) + \int_0^t A^Y_sf - A^X_sf dF_s
\end{align*}
is a local martingale as well. The integral is non-positive and it follows that $(f(X_t) - f(X_0) - f(Y_t) + f(Y_0))_{t \ge 0}$ is a local submartingale. We conclude by Assumption $2.$ that
\begin{align*}
E[f(X_t) - f(Y_t) - f(X_0) + f(Y_0)] \ge 0.
\end{align*}
The assertion follows by Assumptions~$1.$
\end{proof}

The martingale comparison method as developed in this paper allows a comparison of two Markov processes under general alternative regularity conditions compared to the results by the evolution approach.In particular we dismiss with the propagation of order property and also consider the case of random generators which allows for further interesting applications like Markov processes with fixed jump times.

\bibliographystyle{chicago}

\end{document}